\documentclass[letterpaper, reqno, 11pt]{amsart}
\usepackage[english]{babel}
\usepackage{graphicx,color}
\usepackage[all]{xy}
\usepackage{amsfonts,dsfont,mathrsfs}
\usepackage{amsmath, amsthm, amssymb}
\usepackage{enumerate, url}

\newcommand{\bbR}{\ensuremath{\mathbb{R}}}
\newcommand{\bbQ}{\ensuremath{\mathbb{Q}}}
\newcommand{\bbF}{\ensuremath{\mathbb{F}}}
\newcommand{\bbZ}{\ensuremath{\mathbb{Z}}}

\newcommand{\hGamma}{\ensuremath{\widehat{\Gamma}}}
\newcommand{\hphi}{\ensuremath{\widehat{\varphi}}}

\DeclareMathOperator{\ab}{ab}
\DeclareMathOperator{\rk}{rk}

\DeclareMathOperator{\Aut}{Aut}

\DeclareMathOperator{\SL}{SL}

\DeclareMathOperator{\Tor}{Tor}
\DeclareMathOperator{\Div}{Div}
\DeclareMathOperator{\Nil}{Nil}
\DeclareMathOperator{\fin}{fin}

\numberwithin{equation}{section}

\newtheorem{thmnr}{Theorem}[section]

\newtheorem{propnr}[thmnr]{Proposition}

\newtheorem{lemnr}[thmnr]{Lemma}

\newtheorem{cornr}[thmnr]{Corollary}
\theoremstyle{definition}

\newtheorem{dfnnr}[thmnr]{Definition}

\newtheorem{rmknr}[thmnr]{Remark}

\newtheorem{exnr}[thmnr]{Example}

\newtheorem{claimnr}[thmnr]{Claim}

\newtheorem{ques}[thmnr]{Question}

\newtheorem{conjnr}[thmnr]{Conjecture}

\begin{document}

\title{Structure of normally and finitely non-co-Hopfian groups}

\author{Wouter van Limbeek}

\date{\today}

\begin{abstract} A group $G$ is (finitely) co-Hopfian if it does not contain any proper (finite-index) subgroups isomorphic to itself. We study finitely generated groups $G$ that admit a descending chain of proper normal finite-index subgroups, each of which is isomorphic to $G$. We prove that up to finite index, these are always obtained by pulling back a chain of subgroups from a free abelian quotient. We give two applications: First, we show any proper self-embedding of $G$ with finite-index characteristic image arises by pulling back an endomorphism of the abelianization, and secondly, we prove special cases (for normal subgroups) of conjectures of Benjamini and Nekrashevych-Pete regarding the classification of scale-invariant groups. \end{abstract}

\maketitle

\tableofcontents

\section{Introduction}
\label{sec:intro}

\subsection{Main result} A group $\Gamma$ is \emph{co-Hopfian} if it does not contain any proper subgroup isomorphic to itself. A classification of groups that are not co-Hopfian seems extremely difficult because a free group is not co-Hopfian, and similarly, any nontrivial free product $\Gamma=A\ast B$ is not co-Hopfian. Indeed, $\Gamma$ contains the subgroup $A\ast (gBg^{-1})$ which is isomorphic to $A\ast B$ but is a proper subgroup whenever $g\notin A\cup B$. See e.g. \cite{cohopf1,cohopf2,cohopf3} and references therein for further information on and examples of co-Hopfian groups.

Following \cite{fincohopf}, we say a group $\Gamma$ is \emph{finitely co-Hopfian} if $\Gamma$ does not contain any proper finite-index subgroup isomorphic to $\Gamma$. The finite co-Hopf property seems far more tractable than the co-Hopf property in general, at least when $\Gamma$ is finitely generated. Indeed it seems failure of the finite co-Hopf property is closely related to the presence of nilpotent subgroups. De Cornulier has studied which finitely generated nilpotent groups are finitely non-co-Hopfian \cite{cohopfnilp}.

In part this is suggested by an analogous problem of topological nature: Namely, if $M$ is a closed manifold, we say $M$ is \emph{self-covering} if there exists a finite cover $M'\to M$ with degree greater than 1 and such that $M'$ is homeomorphic to $M$. If $M$ is self-covering, then $\pi_1(M)$ is not finitely co-Hopfian. See \cite{selfcover} for more information on self-covering manifolds.

Historically, a situation of particular interest has been the classification of expanding maps. A map $f:M\to M$ is \emph{expanding} if there is a Riemannian metric on $M$ such that $\|Df(v)\|>\|v\|$ for any unit tangent vector $v\in T^1 M$. Franks observed that if $M$ admits an expanding map, then $\pi_1(M)$ has polynomial growth \cite{franksexp}. Gromov proved that a group with polynomial growth is virtually nilpotent, and used this to prove that any manifold admitting an expanding map is infranil \cite{polygrowth}. Monod has informed me by personal communication that the group-theoretic analogue is true: If $\Gamma$ is a finitely generated group that admits an endomorphism $\varphi$ that is expanding with respect to the word metric, in the sense that there exists $L>1$ and $C\geq 0$ such that for all $g,h\in\Gamma$, we have $|d(\varphi(g),\varphi(h))- L d(g,h)|\leq C$, then $\Gamma$ is virtually nilpotent.


Finitely non-co-Hopfian groups have a self-similarity structure:  if $\Gamma_1\subsetneq \Gamma$ is a proper finite-index subgroup isomorphic to $\Gamma$, then by fixing an isomorphism $\Gamma\cong\Gamma_1$, we obtain a self-embedding $\Gamma\hookrightarrow \Gamma$. The images of iterates of this map give a chain
	\begin{equation} \Gamma=\Gamma_0\supsetneq \Gamma_1 \supsetneq \Gamma_2 \supsetneq \dots \label{eq:chain} \end{equation}
of finite-index subgroups of $\Gamma$, each of which is isomorphic to $\Gamma$. 


In this paper we initiate the study of finitely co-Hopfian groups without dynamical conditions on the self-embeddings. Instead, we assume the chain \eqref{eq:chain} consists of normal subgroups. The obvious examples of groups $\Gamma$ admitting such chains are free abelian groups. Our main result is that any example arises in this way:

\begin{thmnr} Let $\Gamma$ be a finitely generated group that admits a decreasing chain of subgroups 
	$$\Gamma=\Gamma_0\supsetneq \Gamma_1 \supsetneq \Gamma_2\supsetneq \dots$$
where $\Gamma_k$ are finite-index normal subgroups with $\Gamma_k\cong\Gamma$. Set $\Gamma_\infty:=\cap_k \Gamma_k$. Then $\Gamma\slash\Gamma_\infty$ is nilpotent and, modulo torsion, it is free abelian. \label{thm:main}
\end{thmnr}
\begin{rmknr} Equivalently, $\Gamma\slash\Gamma_\infty$ is nilpotent and virtually abelian (see Proposition \ref{prop:virt_ab_crit}).
\label{rmk:abmodtor}
\end{rmknr}
In light of Theorem \ref{thm:main}, it is natural to ask if the assumption that the chain \eqref{eq:chain} consists of normal subgroups can be removed:
	\begin{ques} Can one classify finitely non-co-Hopfian groups in terms of nilpotent groups in a manner similar to Theorem \ref{thm:main}?\end{ques}
In any example we know of, a finitely non-co-Hopfian group $\Gamma$ admits a quotient that contains a nontrivial normal nilpotent subgroup. It seems plausible that this is the case for any finitely generated finitely non-co-Hopfian group: It is not difficult to see that any self-embedding $\Gamma\hookrightarrow\Gamma$ induces a self-embedding of the profinite completion $\hGamma$ of $\Gamma$. In this case, work of Reid easily implies that the profinite completion $\widehat{\Gamma}$ contains a nontrivial pronilpotent normal subgroup \cite{reidemb} (see Theorem \ref{thm:reid_emb} and Proposition \ref{prop:con_nilp}). We give examples of $\Gamma$ where the normal nilpotent subgroup can only be found in a quotient and ones where it is not finitely generated (see Example \ref{ex:no_factor}).

Let us now return to the setting of Theorem \ref{thm:main}. It is immediate from Theorem \ref{thm:main} that any chain of subgroups as in Theorem \ref{thm:main} comes from pullback of a chain of subgroups of a virtually abelian group:
\begin{cornr} Let $\Gamma$ be finitely generated and suppose $\Gamma$ admits a decreasing chain of $(\Gamma_k)_{k\geq 0}$ of finite-index normal subgroups with $\Gamma_k\cong\Gamma$. Then there exist
	\begin{itemize}
		\item a finitely generated nilpotent group $N$ such that $N\slash \Tor(N)$ is abelian,
		\item a surjection $\pi:\Gamma\to N$, and
		\item finite-index normal subgroups $N_k\subseteq N$,
	\end{itemize}
such that $\Gamma_k=\pi^{-1}(N_k)$ for every $k\geq 1$.
\label{cor:pullback}
\end{cornr}

Further, it is not necessarily true that $N=\Gamma\slash\Gamma_\infty$ is torsion-free (see Example \ref{ex:tor}). However, in this example the torsion in $N=\Gamma\slash\Gamma_\infty$ already appears in $\Gamma\slash\Gamma_1$ but does not increase thereafter. This is a general phenomenon: Since the torsion subgroup of any finitely generated nilpotent group is finite, and the chain of subgroups $(N_k)_k$ of $N$ in Corollary \ref{cor:pullback} is decreasing with trivial intersection, we see that for $k\gg 1$, the group $N_k$ is free abelian. Hence the torsion can be avoided by passing to a finite-index subgroup of $\Gamma$ that is compatible with the chain of subgroups: 
\begin{cornr} Let $\Gamma$ be a finitely generated group that admits a chain of finite-index normal subgroups $(\Gamma_k)_{k\geq 0}$ with $\Gamma_k\cong\Gamma$. Set $\Gamma_\infty:=\cap_k \Gamma_k$. Then for $l\gg 1$, the group $\Gamma_l\slash\Gamma_\infty$ is free abelian.

In particular $\Gamma_l$ admits a surjection $\pi:\Gamma_l\twoheadrightarrow A$ onto a free abelian group $A$ such that there are finite-index subgroups $A_k\subseteq A$ with $\Gamma_k=\pi^{-1}(A_k)$ for any $k\geq l$.
\label{cor:notor}
\end{cornr}

\begin{rmknr} In general, one may not be able to realize the free abelian group $A$ as a direct factor of any finite-index subgroup of $\Gamma$ (see Example \ref{ex:no_factor}). In that example, $\Gamma$ is given by a semi-direct product $\Delta\rtimes A$, and the self-embedding respects this decomposition. We do not know whether it is always the case that a finite-index subgroup of $\Gamma$ decomposes as a semi-direct product.\end{rmknr}

\begin{rmknr} In \cite{selfcover}, we studied a topological analogue of a special case of Theorem \ref{thm:main}, namely self-covers $f:M'\to M$ of a closed manifold $M$ such that every iterate $f^n:M\to M$ is a regular cover. We proved that any such self-cover is induced by a linear endomorphism $A:T\to T$ of a compact torus $T$ in the following sense: There is a map $h:M\to T$ such that $h\circ f$ is homotopic to $A\circ h$. This is proved using a mixture of the theory of topological transformation groups and structural results for locally finite groups. Unfortunately, neither of these is applicable in the setting of Theorem \ref{thm:main}.\end{rmknr}

\subsection{Characteristic self-embeddings} 


Recall that a subgroup of $\Gamma$ is \emph{characteristic} if it is invariant under all automorphisms of $\Gamma$. In particular, a characteristic subgroup is invariant under all inner automorphisms, and hence it is normal. 

Obvious examples of groups $\Gamma$ with characteristic finite-index subgroups isomorphic to $\Gamma$ are given by free abelian groups. The following result shows that all examples come from free abelian groups.

\begin{cornr} Let $\Gamma$ be a finitely generated group. If $\Gamma$ admits a finite-index characteristic copy $\Gamma'$ of itself, then there is a free abelian group $A$ with characteristic finite-index subgroup $A'$, and a surjection $\pi:\Gamma\to A$ such that $\Gamma'=\pi^{-1}(A')$. In particular, we have $b_1(\Gamma)>0$.
\label{cor:char}
\end{cornr}
\begin{rmknr} Note that Corollary \ref{cor:char} applies as soon as one has a single finite-index characteristic subgroup isomorphic to $\Gamma$, whereas in the setting of normal subgroups, we need a chain for Theorem \ref{thm:main} to hold.\end{rmknr}

\subsection{Scale-invariant groups} Benjamini proposed the following algebraic analogue to the existence of an expanding map on a closed manifold \cite{benjscale}.

\begin{dfnnr}[Benjamini] A finitely generated group $\Gamma$ is \emph{scale-invariant} if there exists a nested decreasing sequence
	$$\Gamma=\Gamma_0\supseteq \Gamma_1 \supseteq \Gamma_2 \supseteq \dots$$
such that $\Gamma_n\cong\Gamma$ for all $n$, and $\cap_n \Gamma_n$ is finite.
\label{dfn:scaleinv}
\end{dfnnr}
\begin{rmknr} The motivation for the introduction of scale-invariant groups comes from percolation theory on graphs: In the traditional setting of percolation, the underlying graph is the grid $\bbZ^d$, and one has access to the very powerful renormalization method. Unfortunately this does not generalize to the Cayley graph of a general group $\Gamma$, but it seems plausible that renormalization techniques  generalize to Cayley graphs of scale-invariant groups. See \cite{scaleinv} for more information on the relationship between percolation theory and scale-invariant groups. \end{rmknr}

Recall that Gromov proved that any closed manifold $M$ admitting an expanding map is infranil, and in particular $\pi_1(M)$ is virtually nilpotent \cite{polygrowth}. Similarly, Benjamini conjectured scale-invariant groups are virtually nilpotent \cite{benjscale}. Many counterexamples to this conjecture were produced by Nekrashevych-Pete \cite{scaleinv}:

\begin{thmnr}[Nekrashevych-Pete \cite{scaleinv}] There exist scale-invariant groups that are not virtually nilpotent. Indeed, if $\Lambda$ is a scale-invariant group with a nested decreasing sequence $(\Lambda_n)_n$ as in Definition \ref{dfn:scaleinv}, and $Q\subseteq \Aut(\Lambda)$ is a group of automorphisms such that each $\Lambda_n$ is $Q$-invariant, then $\Gamma=\Lambda\rtimes Q$ is scale-invariant.

For example, for any subgroup $Q\subseteq \SL(n,\bbZ)$, the group $\bbZ^n\rtimes Q$ is scale-invariant. 
\label{thm:scaleinv}
\end{thmnr}

In these examples, the sequence of subgroups $(\Gamma_n)_n$ is never generated by a single self-embedding, such as an expanding map. To recapture this aspect, Nekrashevych-Pete proposed the following definition:

\begin{dfnnr}[{Nekrashevych-Pete \cite{scaleinv}}] A finitely generated group $\Gamma$ is \emph{strongly scale-invariant} if there exists an embedding $\varphi:\Gamma\hookrightarrow \Gamma$ with image of finite index and such that $\cap_n \, \varphi^n(\Gamma)$ is finite.
\label{dfn:sscaleinv}
\end{dfnnr}
Any strongly scale-invariant group is clearly scale-invariant. However, the counterexamples to Benjamini's conjecture produced in Theorem \ref{thm:scaleinv} are not strongly scale-invariant. Indeed, these groups $\Gamma$ are produced using a self-similar action on a rooted tree, and the subgroups $\Gamma_n$ are constructed by fixing a suitable geodesic ray $c$ and taking $\Gamma_n$ to be the vertex stabilizer at the $n$th level. The self-similarity of the action implies that $\Gamma_n\cong\Gamma$ and since these trees have finite degree, it follows that $[\Gamma:\Gamma_n]<\infty$. In their examples, it is essential that the geodesic ray is aperiodic. This implies that there is no embedding $\varphi:\Gamma\hookrightarrow \Gamma$ such that $\Gamma_n=\varphi^n(\Gamma)$. Based on this observation, Nekrashevych-Pete proposed the following variant of Benjamini's conjecture:

\begin{conjnr}[Nekrashevych-Pete \cite{scaleinv}] Any strongly scale-invariant group is virtually nilpotent.
\label{conj:sscaleinv}
\end{conjnr}

We observe that Theorem \ref{thm:main} implies both Benjamini's conjecture and Nekrashevych-Pete's conjecture \ref{conj:sscaleinv} under the additional assumption that $\Gamma_n$ are all normal:

\begin{cornr} Let $\Gamma$ be a finitely generated group that admits a descending chain $(\Gamma_n)_n$ of finite-index normal subgroups such that $\Gamma_n\cong\Gamma$ and $\cap_n \,\Gamma_n$ is finite. Then $\Gamma$ is virtually abelian.
\label{cor:scaleinv}\end{cornr}

\subsection{Outline of the proof of Theorem \ref{thm:main}} Let $\Gamma$ be a finitely generated group with a descending chain $(\Gamma_k)_k$ of normal finite-index subgroups, each of which is isomorphic to $\Gamma$. We start by observing that the profinite completion $\hGamma$ of $\Gamma$ admits proper open self-embeddings (Proposition \ref{prop:ind_map}). Work of Reid implies that $\hGamma$ contains a normal pronilpotent subgroup, and that each of the finite groups $\Gamma\slash\Gamma_k$ is nilpotent. Most of the proof is devoted to obtaining a uniform bound on the nilpotency class of $\Gamma\slash\Gamma_k$.

To this end, we first prove Theorem \ref{thm:main} assuming $\Gamma$ is nilpotent. This is done in Section \ref{sec:nilp}, and relies on Lie theoretic considerations. 

In the general case, we write the nilpotent group $\Gamma\slash\Gamma_k$ as a product of $p$-groups $(\Gamma\slash\Gamma_k)^{(p)}$ where $p$ runs over the set of primes. Hence the problem of uniformly bounding the nilpotency class separates out into a problem over each prime $p$. In Section \ref{sec:global}, we give a uniform bound at almost every prime. Roughly, the idea is to apply the nilpotent case of Theorem \ref{thm:main} to suitable nilpotent quotients of $\Gamma\slash\Gamma_\infty$ and conclude that modulo torsion, these are abelian. We will show the torsion is concentrated at a finite set of primes, away from which this argument yields the desired uniform bound on nilpotency class.

In Section \ref{sec:local} we bound the nilpotency class of $(\Gamma\slash\Gamma_k)^{(p)}$ for a fixed prime $p$. Combined with the above bound at almost every prime, this yields a uniform bound on the nilpotency class of $\Gamma\slash\Gamma_k$. We finish the proof by once again appealing to the nilpotent case of the main theorem (Section \ref{sec:proofs}). Finally we establish the applications to characteristic finite-index subgroups and scale-invariant groups at the end of Section \ref{sec:proofs}.

\subsection*{Acknowledgments:} I would like to thank Ralf Spatzier for countless helpful discussions. I am thankful to David Fisher for making me aware of the notion of scale-invariant groups and the results of \cite{scaleinv}.

\section{Two examples}
\label{sec:ex}

In this section we discuss two examples. The first shows that in Theorem \ref{thm:main}, the quotient $\Gamma\slash\Gamma_\infty$ is not necessarily abelian or torsion-free.

\begin{exnr} Let $\Delta$ be the three-dimensional integer Heisenberg group, i.e. the group with presentation
	$$\Delta:=\langle x,y,z\mid [x,y]=z, [x,z]=[y,z]=e\rangle.$$
Write $\varphi:\Delta\hookrightarrow \Delta$ for the dilation defined by $\varphi(x):=2x$ and $\varphi(y):=2y$ and $\varphi(z):=4z$. It is easy to see that $\varphi(\Delta)$ is not normal in $\Delta$. However, set $\Lambda:=\langle 2x, 2y, 2z\rangle$. Then $\varphi(\Lambda)\subseteq \Lambda$ is a normal subgroup and $\Lambda\slash\varphi(\Lambda)$ is nonabelian.

Now define $\Gamma:=\Lambda\times \bbZ$ and set $\Gamma_k:=\varphi(\Lambda)\times 2^k \bbZ$ for any $k\geq 1$. Then $(\Gamma_k)_k$ form a chain of finite-index normal subgroups of $\Gamma$, each of which is isomorphic to $\Gamma$, and we have
	$$\Gamma\slash\Gamma_\infty \cong (\Lambda\slash\varphi(\Lambda))\times \bbZ.$$
\label{ex:tor}
\end{exnr} 

In the above example $\Gamma$ has a free abelian factor. However, in general, even though $\Gamma\slash\Gamma_\infty$ is virtually free abelian, we may not be able to realize it as a factor of any finite-index subgroup of $\Gamma$, as the next example shows.

\begin{exnr} Let $d>1$ and consider the corresponding Baumslag-Solitar group 
	$$\Lambda:=\langle x,t\mid txt^{-1} = x^d\rangle.$$
Take two copies $\Lambda_i, i=1,2,$ of $\Lambda$ with generators $x_i, t_i$. Define an automorphism $\varphi_x$ of the free product $\Lambda_1\ast \Lambda_2$ by
	$$\varphi_x(g_i):= x_i g_i x_i^{-1}$$
whenever $g_i\in\Lambda_i$ for $i=1,2$. Likewise define the automorphism $\varphi_t$. Now consider the group
	$$\Gamma:=(\Lambda_1 \ast \Lambda_2)\rtimes_{\varphi_x} \bbZ.$$
Write $y$ for the generator of $\bbZ$ that acts by $\varphi_x$. Finally, define $\psi: \Gamma\hookrightarrow \Gamma$ by $\psi|_{\Lambda_1\ast\Lambda_2}:=\varphi_t$ and $\psi(y):=y^d$. Using that $\varphi_t \varphi_x \varphi_t^{-1}=\varphi_x^d,$
it is easy to verify that $\psi$ is a homomorphism. For any $k\geq 1$, set $\Gamma_k:=\psi^k(\Gamma)$. We have
	$$\Gamma_k=(\Lambda_1\ast\Lambda_2)\rtimes d^k \bbZ\subseteq \Gamma.$$
Hence the group $\Gamma$ admits a decreasing chain of finite-index normal subgroups, each of which is isomorphic to $\Gamma$. However, it is easy to see $\Gamma$ has no free abelian factor.
\label{ex:no_factor}	
\end{exnr}

\section{Preliminaries}
\label{sec:prelim}

\subsection{Profinite groups} We briefly review basic definitions and facts related to profinite groups, and establish notation. For a more thorough discussion, see for example \cite{profin1, profin2}. A group $G$ is \emph{profinite} if it is the inverse limit of an inverse system of finite groups. By equipping each of the finite groups with the discrete topology, we obtain the profinite topology on $G$. 	

%

The profinite topology makes $G$ into a compact Hausdorff space. A local base for the profinite topology of $G$ at the identity $e$ is given by open subgroups of finite index. This gives rise to a useful finiteness property for profinite groups.
	\begin{dfnnr} A profinite group $G$ is said to be \emph{of type (F)} if for every $n\geq 1$, there are only finitely many open subgroups of $G$ of index $n$.
	\label{dfn:typef}
	\end{dfnnr}
If $\Gamma$ is any group, we denote by $\Gamma$ the \emph{profinite completion} of $\Gamma$. There is a natural map $j_\Gamma:\Gamma \to \hGamma$ with dense image. The map $j_\Gamma$ is universal for maps of $\Gamma$ to profinite groups: If $G$ is any profinite group and $\varphi:\Gamma\to G$ is a homomorphism, then there is a unique homomorphism $\hphi:\hGamma\to G$ such that $\varphi=\hphi\circ j_\Gamma$. 

We will be particularly interested in finitely generated groups $\Gamma$. Even though an infinite profinite group is never finitely generated (because it is uncountable), there is a useful concept of finite generation:

\begin{dfnnr} We say a profinite group $G$ is \emph{finitely generated} if $G$ is topologically finitely generated, i.e. there exists a finite set $S\subseteq G$ that generates a dense subgroup of $G$.
\label{dfn:fg_profin}
\end{dfnnr}

Any finitely generated profinite group is of type (F), see e.g. \cite[2.5.1]{profin1}. Further, because the map $j_\Gamma:\Gamma\to\widehat{\Gamma}$ has dense image, we see that if $\Gamma$ is finitely generated, then $\widehat{\Gamma}$ is finitely generated as a profinite group. 

The universal property of the profinite completion has the following consequence in the context of finitely non-co-Hopfian groups:

\begin{propnr} Let $\Gamma$ be a finitely generated group and suppose $\Lambda\subseteq \Gamma$ is a finite-index subgroup of $\Gamma$ with $\Lambda\cong \Gamma$. Fix such an isomorphism and consider the composition 
	$$\varphi:\Gamma\to\Lambda\hookrightarrow \Gamma.$$
Then there exists a unique open embedding $\hphi:\hGamma\to\hGamma$ with $\hphi\circ j_\Gamma = j_\Gamma\circ \varphi$.
\label{prop:ind_map}
\end{propnr}
\begin{proof} By the universal property of profinite completions, the composition
	$$\varphi:\Gamma \overset{\cong}{\longrightarrow} \varphi(\Gamma) \hookrightarrow \Gamma$$
(where the first map is $\varphi$) gives rise to a map
	$$\hphi: \hGamma \overset{\cong}{\longrightarrow} \widehat{\varphi(\Gamma)} \to \hGamma.$$
Here the first map is an isomorphism of profinite groups and in particular a homeomorphism. Therefore we only need to show the second map is an open embedding. But if $G$ is any finitely generated group and $H$ is a finite-index subgroup of $G$, then the map $\widehat{H}\to\widehat{G}$ (induced by inclusion) is an open embedding. This is immediate from the fact that there exists a normal finite-index subgroup $N\subseteq G$ such that $N\subseteq H$.
\end{proof}

\subsection{Contraction groups} We are therefore led to study open self-embeddings of profinite groups. To give a good description of these, we will first recall the notion of a contraction groups, originally introduced by M\"uller-R\"omer \cite{1stcontrgps}. Here we will follow \cite{baum-willis}, who give a slightly different definition that is less general.
\begin{dfnnr} Let $G$ be a locally compact topological group and $\alpha:G\to G$ be an automorphism. Then the pair $(G,\alpha)$ is a \emph{contraction group} if for any $g\in G$, we have $\alpha^n(g)\to e$.
\label{dfn:contr_gp}
\end{dfnnr}

We have the following three essential examples:
\begin{exnr} \mbox{}
	\begin{enumerate}[(1)]
		\item Let $G=\bbR$. Then multiplication by some $\lambda\in\bbR$ with $0<|\lambda|<1$ is a contraction.
		\item Let $G=\bbQ_p$. Then multiplication by $p$ is a contraction.
		\item Let $p$ be a prime and consider the ring of formal Laurent series $G=\bbF_p(\!(t)\!)$ over $\bbF_p$. We view $G$ as an additive group. Then multiplication by $t$ is a contraction.
	\end{enumerate}
	\label{ex:contr}
\end{exnr}
The classification of contraction groups began with the work of Siebert \cite{contrgpclass}, who separated the classification into a problem for connected groups and totally disconnected groups, and completely classified connected contraction groups. In the connected case, $G$ is a simply-connected real unipotent Lie group (compare Example \ref{ex:contr}.(1)).

The totally disconnected case was further developed by the work of Baumgartner-Willis \cite{baum-willis} and solved by Gl\"ockner-Willis \cite{gloeck-willis}. 

\begin{thmnr}[{Gl\"ockner-Willis \cite{gloeck-willis}}] Let $(G,\alpha)$ be a totally disconnected contraction group. Then the set of torsion elements Tor$(G)$ and the set of divisible elements $\Div(G)$ are $\alpha$-invariant closed subgroups of $G$, and
	$$G\cong \Tor(G)\times \Div(G).$$
Further the divisible part is described as follows: there exists a finite set of primes $\{p_i\}_i$ and unipotent $p_i$-adic Lie groups $U_{p_i}$ such that
	$$\Div(G)\cong \prod_i U_{p_i}.$$
\label{thm:contr}
\end{thmnr}

\subsection{Endomorphisms of profinite groups} The above work on contraction groups was used in the context of open self-embeddings of profinite groups by Reid \cite{reidemb}. However, note that if $(G,\alpha)$ is a contraction group, then $G$ is never compact. Therefore we make the following analogous definition in the setting of compact groups:
\begin{dfnnr} Let $G$ be a compact topological group and $\alpha:G\hookrightarrow G$ be a morphism. We say $\alpha$ is \emph{contracting} if for any $g\in G$, we have $\alpha^n(g)\to e$.
\label{dfn:contr_end}
\end{dfnnr}
For the totally disconnected contraction groups of Examples \ref{ex:contr}.(2) and (3), we can take compact subgroups $K$ that are preserved by $\alpha$ and $\alpha$ restricts to a contracting endomorphism of $K$, e.g. $K=\bbZ_p$ in Example \ref{ex:contr}.(2) and $K=\bbF_p[\smallskip[t]\smallskip]$ in Example \ref{ex:contr}.(3). In the case of open contracting embeddings,   it is straightforward that one can also recover the contraction group from the embedding of the compact subgroup:
\begin{propnr} Let $G$ be a compact topological group and $\alpha:G\hookrightarrow G$ be an open contracting embedding. Set
	$$H:=\varinjlim \left(G\overset{\alpha}{\hookrightarrow} G\overset{\alpha}{\hookrightarrow} \dots\right).$$
Then $H$ is a locally compact group and $\alpha$ naturally induces an automorphism of $H$ such that $(H,\alpha)$ is a contraction group.\end{propnr}
\begin{rmknr} Instead of constructing $H$ as an inductive limit, we can also consider the ascending HNN-extension
	$$G\ast_\alpha:= \langle G, t\mid \forall g\in G: tgt^{-1}=\alpha(g) \rangle.$$
The group $H$ is then obtained as
	$$H=\bigcup_{n\geq 0} t^{-n} G t^n.$$
\end{rmknr}

Using the theory of contraction groups, we have the following result by Reid that describes open embeddings of finitely generated profinite groups.

\begin{thmnr}[{Reid \cite{reidemb}}] Let $G$ be a profinite group of type (F) and $\hphi:G\hookrightarrow G$ an open embedding. Then there exist closed subgroups $C$ and $Q$ of $G$ such that
	\begin{itemize}
		\item $G\cong C\rtimes Q$,
		\item $C$ is $\hphi$-invariant and $\hphi$ restricts to an open contracting embedding on $C$,
		\item We can write $C=N\times F$ where $N$ is a compact open subgroup of a product of finitely many $p$-adic unipotent Lie groups (for finitely many primes $p$) and $F$ is a bounded exponent solvable group that is residually nilpotent, and
		\item $Q$ is $\hphi$-invariant and $\hphi$ restricts to an automorphism on $Q$.
	\end{itemize}
\label{thm:reid_emb}
\end{thmnr}

\begin{rmknr} $C$ and $Q$ are explicitly given in terms of $\hphi$, namely we have
	$$C:=\{g\in G\mid \hphi^n(g)\overset{n\to\infty}{\longrightarrow} e\}$$
and $Q:=\displaystyle\bigcap_{n\geq0} \hphi^n(G).$
\label{rmk:cq_expl}
\end{rmknr}

In fact Reid proves a more general version of the decomposition of Theorem \ref{thm:reid_emb}, where one can considers a collection of endomorphisms $G\to G$. We will need a stronger conclusion than Theorem \ref{thm:reid_emb} provides in the restricted case of a single endomorphism, which however is immediate from the proof of Theorem \ref{thm:reid_emb} in \cite{reidemb}. 

To state this stronger version, recall that a group $G$ is \emph{pronilpotent} if it is the inverse limit of a system of finite nilpotent groups. In particular, any pronilpotent group is residually nilpotent. The stronger version of Theorem \ref{thm:reid_emb} that we need is as follows.

\begin{propnr} Let $G$ be a profinite group of type (F) and $\hphi:G\hookrightarrow G$ an open embedding. Let $C$ be as constructed in Theorem \ref{thm:reid_emb}. Then $C$ is pronilpotent.
\label{prop:con_nilp}
\end{propnr}
\begin{proof} This slightly stronger result is immediate from Reid's proof of \cite[Theorem 4.3.(i)]{reidemb}. Indeed, after the reduction to the case $G=\text{Con}(\hphi)$, Reid proceeds to construct a sequence of finite-index open normal subgroups $N_i$ with $\cap N_i=1$ and such that $G\slash N_i$ is nilpotent for each $i$. Hence $G$ is pronilpotent.\end{proof}

\subsection{Nilpotent groups} The previous result leads us from open embeddings of profinite groups to pronilpotent groups. To exploit this connection later, we will now recall some classical facts about nilpotent groups. Recall that the \emph{lower central series} of a group $G$ is inductively defined by $\gamma_0(G):=G$ and $\gamma_{k+1}(G):=[\gamma_k(G),G]$ for any $k\geq 0$. We say $G$ is \emph{nilpotent of class} $c$ if $c$ is the first term of the lower central series such that $\gamma_c(G)=1$. We say $G$ is \emph{nilpotent} if $G$ is nilpotent of class $c$ for some $c$.

We start by recalling the following elementary result that gives a splitting of a finite nilpotent group over primes.
	\begin{propnr}[{See e.g. \cite[5.2.4]{rob_gps}}] Let $G$ be a finite nilpotent group. Then
		\begin{enumerate}[(i)]
			\item For any prime $p$ dividing $|G|$, there exists a unique $p$-Sylow subgroup $G^{(p)}$ of $G$, and
			\item We have $G\cong \prod_p G^{(p)}$, where the product is over all primes.
		\end{enumerate}
		\label{prop:nilp_prod}
	\end{propnr}
The corresponding result holds for profinite groups:
	\begin{propnr}[{see e.g. \cite[2.3.8]{profin1}}]  Let $G$ be a pronilpotent group. Then for any prime $p$, there exist closed pro-$p$ subgroups $G^{(p)}$ of $G$ such that
		$$G\cong \prod_p G^{(p)},$$
	where the product is over all primes.
		\label{prop:pronilp_prod}
	\end{propnr}
In the above setting, $G^{(p)}$ is called the \emph{localization of} $G$ \emph{at the prime} $p$. We continue by studying finitely generated nilpotent groups. The following result allows us to reduce to the torsion-free case.
	\begin{propnr}[{See e.g. \cite[5.2.7]{rob_gps}}] Let $\Gamma$ be a finitely generated nilpotent group. Then the set of torsion elements $\Tor(\Gamma)$ is a finite normal subgroup of $\Gamma$.
	\label{prop:tor_subgp}
	\end{propnr}
Therefore the study of a finitely generated nilpotent group $\Gamma$ breaks up into studying its torsion subgroup $\text{Tor}(\Gamma)$ and the torsion-free part $\Gamma^{(\infty)}:=\Gamma\slash\text{Tor}(\Gamma)$. Occasionally it will be useful to us to work with a finite-index torsion-free subgroup rather than the torsion-free quotient $\Gamma^{(\infty)}$. We can do so by the following result, which follows immediately from Proposition \ref{prop:tor_subgp} and the fact that nilpotent groups are residually finite, a result of Gruenberg (see e.g. \cite[5.2.21]{rob_gps}):
	\begin{propnr} Let $\Gamma$ be a finitely generated nilpotent group. Then $\Gamma$ is virtually torsion-free. In particular, $\Gamma\slash\Tor(\Gamma)$ is abelian if and only if $\Gamma$ is virtually abelian.
	\label{prop:virt_ab_crit}
	\end{propnr}
The theory of finitely generated torsion-free nilpotent groups is intimately connected to Lie theory because of the following result of Mal'cev.
	\begin{thmnr}[{Mal'cev (see e.g. \cite[Cor. 2 of 2.11 and 2.18]{raghlie})}] Let $\Gamma$ be a finitely generated torsion-free nilpotent group. Then there exists a unique simply-connected nilpotent Lie group $N$ such that there is an embedding $\Gamma\hookrightarrow N$ as a cocompact lattice.
	\label{thm:nilp_hull}
	\end{thmnr}
Given a finitely generated torsion-free nilpotent group $\Gamma$, we say the simply-connected nilpotent Lie group $N$ given by Theorem \ref{thm:nilp_hull} is the \emph{real Mal'cev completion} of $\Gamma$ and use the notation $N=\Gamma\otimes\bbR$. The real Mal'cev completion captures an incredible amount of the algebraic structure of $\Gamma$. One of the strongest instances is the following superrigidity result of Mal'cev, showing that homomorphisms of $\Gamma$ to other simply-connected nilpotent Lie groups are controlled by $N$.
	\begin{thmnr}[{Mal'cev (see e.g. \cite[2.17]{raghlie})}] Let $\Gamma$ be a finitely generated torsion-free nilpotent group. Then for any simply-connected nilpotent Lie group $H$ and any homomorphism $\rho:\Gamma \to H$, the map $\rho$ uniquely extends to a homomorphism $\overline{\rho}:\Gamma\otimes\bbR\to H$.
	\label{thm:super_malcev}
	\end{thmnr}
	
\section{Nilpotent case of the main theorem}
\label{sec:nilp}

The goal of this section is to prove Theorem \ref{thm:main} under the additional assumption that $\Gamma$ is nilpotent, i.e. the following result.

\begin{thmnr} Suppose $\Gamma$ is a finitely generated and nilpotent group and admits a descending chain of finite-index normal subgroups $\Gamma_k$ with $\Gamma_k\cong \Gamma$. Set $\Gamma_\infty=\cap_k \Gamma_k$. Then modulo its torsion, $\Gamma\slash\Gamma_\infty$ is abelian.
\label{thm:nilp_case}
\end{thmnr}
This result will be used several times in the subsequent sections to establish the general version of Theorem \ref{thm:main}. Before starting the proof, we provide a brief outline: 

First we reduce the statement to the case of torsion-free nilpotent groups (Step 1 below). Any such group is a lattice in its real Mal'cev completion $G$. Using that the real Mal'cev completion $G$ is the same for $\Gamma$ and its subgroup $\Gamma_k$, we obtain automorphisms $\varphi_k$ of $G$ with $\varphi_k(\Gamma)=\Gamma_k$ (Step 2).

Then $\Gamma_{-k}:=\varphi_k^{-1}(\Gamma)$ are cocompact lattices in $G$ contained in $H:=\overline{\cup_k \Gamma_{-k}}$. This allows us to consider an (archimedean) limit of $\Gamma\slash\Gamma_k \cong \Gamma_{-k}\slash\Gamma$, namely $H\slash\Gamma$. Contrary to the profinite limit $\varprojlim \Gamma\slash\Gamma_k$, we can control $H\slash\Gamma$ because it is a compact nilpotent Lie group. We use this to show that $\Gamma\slash\Gamma_k$ have normal abelian subgroups of uniformly bounded index (Step 3). Finally we conclude from this that $\Gamma\slash\Gamma_\infty$ is abelian modulo torsion (Step 4).

\begin{proof}[{Proof of Theorem \ref{thm:nilp_case}}] \mbox{}

\subsection*{Step 1 (reduction to torsion-free case)} 
We claim that it suffices to establish the theorem for torsion-free nilpotent groups. Indeed, let $\Gamma$ be any finitely generated group with a descending chain of finite-index normal subgroups $\Gamma_k$ with $\Gamma_k\cong\Gamma$. Choose such isomorphisms and view these as embeddings $\varphi_k:\Gamma\hookrightarrow \Gamma$ (with image $\Gamma_k$). Recall that the set $\Tor(\Gamma)$ of $\Gamma$ form a finite normal subgroup, and likewise for $\Gamma_k$ (see Proposition \ref{prop:tor_subgp}).

For any $k\geq 1$, we clearly have
	\begin{equation}
	\varphi_k(\Tor(\Gamma))\subseteq \Tor(\Gamma_k)\subseteq \Tor(\Gamma). 
	\label{eq:tor_incl}
	\end{equation}
On the other hand, since $\Gamma_k\cong\Gamma$, we have $\Tor(\Gamma_k)\cong\Tor(\Gamma)$. Since $\Tor(\Gamma)$ is finite, both inclusions in Equation \ref{eq:tor_incl} are equalities. Hence $\varphi_k:\Gamma\hookrightarrow\Gamma$ descends to an embedding $\varphi_k:\Gamma\slash \Tor(\Gamma)\to \Gamma\slash \Tor(\Gamma)$. Further note that $\Gamma^{(\infty)}:=\Gamma\slash \Tor(\Gamma)$ is torsion-free and nilpotent, and if it is abelian modulo its torsion, then so is $\Gamma$. This completes Step 1.

For the remainder of the proof, we assume $\Gamma$ satisfies the hypotheses of Theorem \ref{thm:nilp_case} and is torsion-free. We choose isomorphisms $\varphi_k$ as above, and view them as self-embeddings $\varphi_k :\Gamma\hookrightarrow \Gamma$ with image $\Gamma_k$.

\subsection*{Step 2 (constructing automorphisms of the Mal'cev completion)} Let $G:=\Gamma\otimes\bbR$ be the real Mal'cev completion of $\Gamma$, so that $\Gamma$ embeds as a cocompact lattice in $G$ (see Theorem \ref{thm:nilp_hull}). Henceforth we will identify $\Gamma$ with its image in $G$. 

By Mal'cev's Superrigidity Theorem \ref{thm:super_malcev}, the embeddings $\varphi_k:\Gamma\hookrightarrow \Gamma$ uniquely extend to continuous homomorphisms $G\to G$, which we will also denote by $\varphi_k$. We claim each $\varphi_k$ is an automorphism of $G$. 

Indeed, since $\varphi_k(\Gamma)$ is of finite index in $\Gamma$, we know that $\varphi_k(\Gamma)$ is also a lattice in $G$. Therefore we can also apply Mal'cev's Superrigidity Theorem \ref{thm:super_malcev} to the inverse map
	$$\varphi_k^{-1}|_{\varphi_k(\Gamma)}: \varphi_k(\Gamma)\to \Gamma.$$
The uniqueness in Mal'cev's superrigidity theorem implies that the extension of the inverse is inverse to $\varphi_k:G\to G$. Hence $\varphi_k$ extends to an automorphism of $G$. 

\subsection*{Step 3 (control on $\Gamma\slash\Gamma_k$)} We claim that there exists $C\geq 1$ such that for any $k\geq 1$, the group $\Gamma\slash\Gamma_k$ has a normal abelian subgroup of index at most $C$. 

To see this, set $\Gamma_{-k}:=\varphi_k^{-1}(\Gamma)$ (as a subgroup of $G$), where $k\geq 1$. Set
		$$H:=\overline{\bigcup_{k\geq 0} \Gamma_{-k}}.$$
Since $\Gamma\subseteq G$ is cocompact and is contained in the closed subgroup $H\subseteq G$, we must have that $\Gamma$ is also cocompact in $H$. Further since $\Gamma_k$ is a normal subgroup of $\Gamma$ for any $k\geq 1$, we have that $\Gamma$ is a normal subgroup of $\Gamma_{-k}$ for any $k\geq 1$. Since $\Gamma$ is closed, it follows that $H$ also normalizes $\Gamma$. 

Hence $K:=H\slash\Gamma$ is a compact nilpotent Lie group containing $\Gamma_{-k}\slash\Gamma$ for any $k\geq 1$. Note that $K$ has finitely many connected components and its identity component $K^0$ is a compact connected nilpotent group and hence is a torus. Let $C:=|K\slash K^0|$ be the number of connected components of $K$. 

Let $k\geq 1$ and consider the composition
	$$f: \Gamma\slash\Gamma_k \overset{\varphi_k^{-1}}{\longrightarrow} \Gamma_{-k}\slash\Gamma \subseteq H\slash \Gamma =K \twoheadrightarrow K\slash K^0.$$
Note that $\ker(f)$ has index at most $C$. Further $\ker(f)$ embeds into the compact torus $K^0$ (by omitting the last map in the above composition) and hence is abelian. This proves the claim.

\subsection*{Step 4 (end of the proof)} We claim that $\overline{\Gamma}:=\Gamma\slash\Gamma_\infty$ is virtually abelian, which will finish the proof (using Proposition \ref{prop:virt_ab_crit}). Let $C\geq 1$ be as in Step 3. Since $\Gamma$ is finitely generated, it has only finitely many subgroups of index at most $C$. Define $\Lambda$ to be the intersection of all of these subgroups. We will show that the image of $\Lambda$ in $\overline{\Gamma}$ is abelian.

To see this, note that $\Lambda$ is a finite-index normal subgroup of $\Gamma$ with the following property: Whenever $L$ is a finite group of order at most $C$ and $f:\Gamma\to L$ is a homomorphism, $\Lambda\subseteq \ker(f)$. 

Hence for any $k\geq 1$, the image of the composition
	$$\Lambda\to \overline{\Gamma}\to \Gamma\slash\Gamma_k \overset{\varphi_k^{-1}}{\longrightarrow} \Gamma_{-k}\slash\Gamma\hookrightarrow K$$
is contained in $K^0$ (because $\Lambda$ maps trivially to $K\slash K^0$). We conclude that the image of $\Lambda$ in $\Gamma\slash\Gamma_k$ is abelian, so that $[\Lambda,\Lambda]\subseteq \Gamma_k$ for all $k$. Hence 
	$$[\Lambda,\Lambda]\subseteq \cap_k \Gamma_k = \Gamma_\infty.$$ 
Therefore the image of $\Lambda$ in $\overline{\Gamma}$ is abelian, as desired.\end{proof}

\section{Almost global bound on nilpotency class}
\label{sec:global}

We will now start the proof of the general case of Theorem \ref{thm:main}. Let $\Gamma$ be a finitely generated group with a descending chain of finite-index normal subgroups $\{\Gamma_k\}_k$ such that $\Gamma_k\cong\Gamma$. We start by using the structure of the profinite completion of $\Gamma$ to obtain information about the finite groups $\Gamma\slash\Gamma_k$.

\begin{propnr} $\Gamma\slash\Gamma_k$ is nilpotent for every $k\geq 1$. \end{propnr}
\begin{proof} Fix isomorphisms $\Gamma\cong\Gamma_k$ and view these as self-embeddings
	$$\varphi_k:\Gamma\hookrightarrow \Gamma$$
with image $\Gamma_k$. These maps induce
	$$\hphi_k:\hGamma\to \hGamma,$$
on the profinite completion of $\Gamma$. By Proposition \ref{prop:ind_map}, the map $\hphi_k$ is an open embedding for each $k$. By Reid's work on open self-embeddings of profinite groups (see Theorem \ref{thm:reid_emb}), there are closed subgroups $C_k\subseteq \hGamma$ and $Q_k\subseteq \hGamma$ such that
	$$\hGamma=C_k\rtimes Q_k$$
and $\hphi_k$ restricts to a contracting endomorphism of $C_k$ and an automorphism of $Q_k$. In particular, we have
	$$\Gamma\slash\Gamma_k \cong \hGamma\slash\hphi_k(\hGamma) = C_k \slash \hphi_k(C_k).$$
By Proposition \ref{prop:con_nilp}, the group $C_k$ is pronilpotent, and since $\hphi(C_k)\subseteq C_k$ is open, it follows that the finite quotient $C_k\slash\hphi_k(C_k)\cong \Gamma\slash\Gamma_k$ is nilpotent.\end{proof}



In order to use the nilpotency of $\Gamma\slash\Gamma_k$ to obtain information about $\Gamma$, we will show the nilpotency class of $\Gamma\slash\Gamma_k$ is uniformly bounded. Most of the next two sections is devoted to the proof of this. Once this uniform bound has been established, the proof of Theorem \ref{thm:main} will be finished in Section \ref{sec:proofs}.  

To bound the nilpotency class of $\Gamma\slash\Gamma_k$, write it as a product over primes (see Proposition \ref{prop:nilp_prod}):
	$$\Gamma\slash\Gamma_k = \prod_p (\Gamma\slash\Gamma_k)^{(p)}$$
where $(\Gamma\slash\Gamma_k)^{(p)}$ is a finite $p$-group. In this way, the problem of uniformly bounding the nilpotency class of $\Gamma\slash\Gamma_k$ separates out into a problem over each prime. The desired uniform bound on nilpotency class will be obtained in two steps: First we give a uniform bound on the nilpotency class of $(\Gamma\slash\Gamma_k)^{(p)}$ for almost all primes $p$ (see Theorem \ref{thm:almost_global} below). Finally, in Theorem \ref{thm:local}, we will obtain a local bound on nilpotency class at any fixed prime $p$. The rest of this section will be devoted to obtain the bound at almost all primes.

\begin{thmnr} There is a finite set of primes $S$ and $c\geq 1$ such that for any $p\notin S$ and $k\geq 1$, the group $(\Gamma\slash \Gamma_k)^{(p)}$ is nilpotent of class at most $c$.
\label{thm:almost_global}
\end{thmnr}

Let us first introduce some notation. Recall that $\gamma_c(\Gamma)$ is the $c$th term in the lower central series of $\Gamma$. Set $N_c:=\Gamma\slash\gamma_c(\Gamma)$. Note that $N_c$ is a finitely generated nilpotent group of class at most $c$. Set $N_c^{(\infty)}=N_c\slash \Tor(N_c)$, so $N_c^{(\infty)}$ is a finitely generated torsion-free nilpotent group.

Let us now give an outline of the proof of Theorem \ref{thm:almost_global}. We will first show that $\varphi_k:\Gamma\hookrightarrow \Gamma$ descend to self-embeddings of $N_c^{(\infty)}$ (Claims \ref{cl:embed} and \ref{cl:index_div}). We can therefore apply the nilpotent case of the Main Theorem \ref{thm:main}, i.e. Theorem \ref{thm:nilp_case}. This easily shows that any nonabelian subgroup of $\Gamma\slash\Gamma_k$ comes from the torsion of a certain nilpotent group of class $c$, where $c$ is the nilpotency class of $\Gamma\slash\Gamma_k$. We will show that if a prime contributes nonabelian torsion at some class $c_0$, then it already did so at $c=2$ (Lemma \ref{lem:ab_or_div}), and hence nonabelian subgroups of $\Gamma\slash\Gamma_k$ are only located at the finite set of primes visible at $c=2$. We will use this to complete the proof of Theorem \ref{thm:almost_global} at the end of the section.

\begin{claimnr} $\varphi_k:\Gamma\hookrightarrow \Gamma$ descend to embeddings $\overline{\varphi}_k^{(\infty)}:N_c^{(\infty)}\hookrightarrow N_c^{(\infty)}$.
\label{cl:embed}
\end{claimnr}
\begin{proof} 
It is clear each $\varphi_k$ descends to a map
	$$\overline{\varphi}_k:N_c\to N_c$$
such that the image $\overline{\varphi}_k(N_c)$ is normal and finite-index in $N_c$. Since the image of any torsion element is torsion, we see that each $\overline{\varphi}_k$ in turn descends to a map
	$$\overline{\varphi}_k^{(\infty)}:N_c^{(\infty)}\to N_c^{(\infty)}$$
with finite-index and normal image. It remains to show $\overline{\varphi}_k^{(\infty)}$ is injective. 

Let $G=N_c^{(\infty)}\otimes\bbR$ be the real Mal'cev completion of $N_c^{(\infty)}$, i.e. the unique simply-connected nilpotent Lie group containing $N_c^{(\infty)}$ as a cocompact lattice (see Theorem \ref{thm:nilp_hull}). By Mal'cev's Superrigidity Theorem \ref{thm:super_malcev}, the maps $\overline{\varphi}_k^{(\infty)}$ extend to 
continuous homomorphisms
	$$\overline{\varphi}_k^{(\infty)}:G\to G.$$
We claim these are isomorphisms. Indeed, $\overline{\varphi}_k^{(\infty)}(N_c^{(\infty)})$ is a finite-index subgroup of $N_c^{(\infty)}$ and hence a cocompact lattice in $G$. Therefore $\overline{\varphi}_k(G)$ is a connected closed cocompact subgroup of $G$. But since $G$ is a simply-connected nilpotent Lie group, the only connected closed cocompact subgroup of $G$ is $G$ itself (see e.g. \cite[2.1]{raghlie}). It follows that $\overline{\varphi}_k^{(\infty)}$ is surjective. 

By a dimension count, we see that $\dim \ker\overline{\varphi}_k^{(\infty)}=0$, i.e. $\ker\overline{\varphi}_k^{(\infty)}$ is discrete, so that $\overline{\varphi}_k^{(\infty)}$ is a covering map. Since $\overline{\varphi}_k^{(\infty)}(G)=G$ is simply-connected, it follows that $\overline{\varphi}_k^{(\infty)}$ is an automorphism of $G$, as desired. \end{proof}

Consider the set $\mathscr{P}$ of primes $p$ such that for some $k\geq 1$, we have $(\Gamma\slash\Gamma_k)^{(p)}\neq 1$. If $\mathscr{P}$ is finite, then the conclusion of Theorem \ref{thm:almost_global} is immediately satisfied. Henceforth we assume $\mathscr{P}$ is infinite.

\begin{claimnr} For any $c\geq 1$, we have $[N_c^{(\infty)}:\overline{\varphi}_k^{(\infty)}(N_c^{(\infty)})]\overset{k\to\infty}{\longrightarrow}\infty.$
\label{cl:index_div}\end{claimnr}

\begin{rmknr} At this point we do not even know that $N_c$ is infinite, but this will follow from the above claim.\end{rmknr}

\begin{proof} Since the map $N_c\to N_c^{(\infty)}$ has finite kernel $\Tor(N_c)$, we have for any $c$ and $k$ that
	$$[N_c^{(\infty)}:\overline{\varphi}_k^{(\infty)}(N_c^{(\infty)})]\geq \frac{1}{|\Tor(N_c)|} [N_c:\overline{\varphi}_k(N_c)].$$
Therefore it suffices to show that for any $k\geq 1$, we have 
	\begin{equation}
	[N_c:\overline{\varphi}_k(N_c)]\overset{k\to\infty}{\longrightarrow}\infty.
	\label{eq:index_div}
	\end{equation}
Further, it suffices to establish \eqref{eq:index_div} for $c=1$ because the abelianization map $\ab: N_c\to N_1=\Gamma^{\ab}$ satisfies $\ab\left(\overline{\varphi}_k(N_c)\right)=\overline{\varphi}_k(N_1),$ so that
	$$[N_c:\overline{\varphi}_k(N_c)]\geq [N_1:\overline{\varphi}_k(N_1)].$$
It remains to show that $[N_1:\overline{\varphi}_k(N_1)]\to\infty$. It is straightforward to see that
	$$N_1\slash\overline{\varphi}_k(N_1) = (\Gamma\slash\Gamma_k)^{\ab}$$
so $[N_1:\overline{\varphi_k}(N_1)]=|(\Gamma\slash\Gamma_k)^{\ab}|.$ Let $\mathscr{P}_k$ be the set of prime divisors of $|\Gamma\slash\Gamma_k|$. Since $\Gamma\slash\Gamma_k$ splits as a product of nontrivial $p$-groups for $p\in\mathscr{P}_k$, and any nontrivial $p$-group has nontrivial abelianization, we find
	$$|(\Gamma\slash\Gamma_k)^{\ab}|\geq \prod_{p\in\mathscr{P}_k} p.$$
By assumption, $\mathscr{P}=\cup_k \mathscr{P}_k$ is infinite, so we have $|(\Gamma\slash\Gamma_k)^{\ab}|\to\infty$, as desired.\end{proof}

Since $[N_c^{(\infty)}:\overline{\varphi}_k^{(\infty)}(N_c^{(\infty)})]\to\infty$, we can assume (after passing to a subsequence if necessary) that 
	$$N_c^{(\infty)}\supseteq \overline{\varphi}_1^{(\infty)}(N_c^{(\infty)})\supseteq \overline{\varphi}_2^{(\infty)}(N_c^{(\infty)})\supseteq \dots$$
is a descending chain of subgroups. By Theorem \ref{thm:nilp_case} (the nilpotent case of Theorem \ref{thm:main}), we see that $\overline{N_c^{(\infty)}}:=N_c^{(\infty)}\slash\cap_k \overline{\varphi}_k^{(\infty)}(N_c^{(\infty)})$ is abelian modulo torsion. In fact, we can conclude the following slightly stronger result.

\begin{lemnr} $\overline{N}_c:=N_c\slash\cap_k \overline{\varphi}_k(N_c)$ is abelian modulo its torsion. 
\label{lem:ab_mod_tor}
\end{lemnr}

\begin{proof} Taking the quotient by Tor$(N_c)$ gives a surjective map 
	$$\pi_c^{(\infty)}:N_c\to N_c^{(\infty)}$$
that maps $\cap_k \overline{\varphi}_k(N_c)$ to $\cap_k \overline{\varphi}_k^{(\infty)}(N_c^{(\infty)})$. Hence $\pi_c^{(\infty)}$ descends to a surjective map
	$$\overline{N_c} \to \overline{N_c^{(\infty)}}$$
with finite kernel Tor$(N_c)\slash (\text{Tor}(N_c)\cap \cap_k \overline{\varphi}_k(N_c))$. Since $\overline{N_c^{(\infty)}}$ is abelian modulo its torsion, the same holds for $\overline{N_c}$.\end{proof}
Next we show the torsion-free part $\overline{N}_c^{(\infty)}$ of $\overline{N}_c$ does not depend on $c$.
	\begin{lemnr} Let $c>d\geq 1$. Then the natural maps $\overline{\pi}_c:\overline{N}_c\to\overline{N}_{d}$ descend to isomorphisms
		$$\overline{\pi}_c^{(\infty)}:\overline{N}_c^{(\infty)}\overset{\cong}{\longrightarrow} \overline{N}_{d}^{(\infty)}.$$
In particular, $\ker(\overline{\pi}_c)$ is torsion.
	\label{lem:tor_ker}
	\end{lemnr}
	\begin{proof} Surjectivity of the quotient map $N_c\to N_{d}$ immediately yields that $\overline{\pi}_c$ is surjective. To show it is injective, it suffices to consider the case $d=1$. Consider the composition
		\begin{equation} \Gamma \to \overline{N}_c\to \overline{N}_c^{(\infty)}.\label{eq:factor}\end{equation}
	Since $\Gamma_\infty$ is contained in the kernel of \eqref{eq:factor} and $\overline{N}_c^{(\infty)}$ is torsion-free abelian, we see \eqref{eq:factor} factors through the quotient $\Gamma\to \overline{N}_1^{(\infty)}$ and hence descends to a map $\overline{N}_1^{(\infty)}\to\overline{N}_c^{(\infty)}$, which is easily checked to be inverse to $\overline{\pi}_c^{(\infty)}$. \end{proof}
%
%

Our next step is most conveniently phrased in terms of subvarieties of groups and verbal subgroups associated to the lower central series. Let us therefore first recall these notions and fix some notation.
\begin{dfnnr} Let $V$ be a collection of words on some collection of symbols. Let $V(G)$ denote the corresponding verbal subgroup of $G$, and write $\mathfrak{V}(G):=G\slash V(G)$ for the largest quotient of $G$ that satisfies all the group laws given by $V$.\end{dfnnr}
Note that terms of the lower central series of a group give examples of verbal subgroups. For the $c$th term, we can consider the verbal subgroup $V(G)=\gamma_d(G)$, and we will write $\Nil_c(G):=\mathfrak{V}(G)=G\slash\gamma_d(G)$. For example, in the special case $G=\Gamma$ we have been writing $N_c=\Nil_c(\Gamma)$. We need the following general fact, the proof of which is straightforward.
\begin{lemnr} Let $G$ be any group and suppose $H\subseteq G$ is a normal subgroup. Write $\mathfrak{V}_G(H)$ for the image of $H$ in $\mathfrak{V}(G)$. Then 
	$$\mathfrak{V}(G)\slash \mathfrak{V}_G(H)\cong \mathfrak{V}(G\slash H).$$
\label{lem:verbal}
\end{lemnr}
Our goal will now be to show that any nonabelian subgroup of $\Gamma\slash\Gamma_k$ is located at divisors of $\Tor(\overline{N}_2)$. More precisely, if $\Gamma\slash\Gamma_k$ is nilpotent of class $c$, we have $\Gamma\slash\Gamma_k = N_c\slash\overline{\varphi}_k(N_c)$. The latter surjects onto $N_c^{(\infty)}\slash\overline{\varphi}_k^{(\infty)}(N_c^{(\infty)})$, which is abelian by Lemma \ref{lem:ab_mod_tor}. Therefore any nonabelian subgroup of $\Gamma\slash\Gamma_k$ is located at divisors of $|\Tor(N_c)|$. Our goal is to show such a subgroup in fact needs to be located at divisors of $|\Tor(\overline{N}_2)|$.

Unfortunately in general the torsion of $N_c$ can grow as $c\to\infty$. We remedy this in the following way. Since $\overline{N}_c$ is finitely generated and nilpotent, $\overline{N}_c$ is virtually torsion-free (See Proposition \ref{prop:virt_ab_crit}). Let $\overline{N}_c^{(\text{tf})}$ be a finite-index normal torsion-free subgroup of $\overline{N}_c$. Write 
	$$\overline{N}_c^{(\fin)}:=\overline{N}_c\slash\overline{N}_c^{(\text{tf})}$$
for the quotient. Set $\overline{T}_c:=\Tor(\overline{N}_c)$. Since $\overline{N}_c^{(\text{tf})}$ is torsion-free, $\overline{T}_c$ projects isomorphically onto its image in $\overline{N}_c^{(\fin)}$.

Consider the image of $\overline{N}_c^{(\text{tf})}$ under $\overline{\pi}_d:\overline{N}_c\to \overline{N}_2$. Since $\ker\overline{\pi}_d$ is torsion (by Lemma \ref{lem:tor_ker}), we see that $\overline{N}_2^{(\text{tf})}:=\overline{\pi}_d(\overline{N}_c^{(\text{tf})})$ is a torsion-free finite-index normal subgroup of $\overline{N}_2$. Write
	$$\overline{N}_2^{(\fin)}:=\overline{N}_2\slash\overline{N}_2^{(\text{tf})}$$
for the quotient. Unlike the situation for torsion groups of $\overline{N}_c$ at different values of $c$, there is an easy description of the relationship between $\overline{N}_c^{(\fin)}$ and $\overline{N}_2^{(\fin)}$ as follows.

\begin{lemnr} We have
	\begin{enumerate}[(i)]
		\item $\overline{N}_c=\Nil_c(\overline{\Gamma})$, and
		\item $\Nil_2\left(\overline{N}_c^{(\fin)}\right)\cong \overline{N}_2^{(\fin)}.$
	\end{enumerate}
 \label{lem:nil2}
 \end{lemnr}
\begin{proof} We apply Lemma \ref{lem:verbal} to the group $\Gamma$ with normal subgroup $\Gamma_\infty$ and verbal subgroup $\gamma_c(\Gamma)$. Since $\cap_k\, \overline{\varphi}_k(N_c)$ is the image of $\Gamma_\infty$ in $N_c$, we obtain that
	$$\overline{N}_c=N_c\slash\cap_k \overline{\varphi}_k(N_c) = \Nil_c(\Gamma\slash\Gamma_\infty).$$
This proves (i).

It follows that $\Nil_2(\overline{N}_c)=\overline{N}_2$. Applying Lemma \ref{lem:verbal} again to the group $\overline{N}_c$ with normal subgroup $\overline{N}_c^{(\text{tf})}$ and verbal subgroup given by the second term $\gamma_2$ of the lower central series (so $\mathfrak{V}(G)=\Nil_2(G)$ for any group $G$), we find
	$$
		\overline{N}_2\slash \overline{N}_2^{(\text{tf})} \cong \Nil_2(\overline{N}_c\slash \overline{N}_c^{(\text{tf})}).
	$$
\end{proof}

%

Write 
	$$\overline{N}_c^{(\fin)}=\prod_p \overline{N}_c^{(\fin,p)}$$
as a product of $p$-groups. We show that any nonabelian contributions are located at a finite set of primes that does not depend on $c$:

\begin{lemnr} Let $p$ be a prime that does not divide $|\overline{T}_2|$. Then $\overline{N}_c^{(\fin,p)}$ is abelian.
\label{lem:ab_or_div}
\end{lemnr}

\begin{proof} 

%




By Lemma \ref{lem:nil2}.(ii), we have $\overline{N}_2^{(\fin,p)}=\Nil_2(\overline{N}_c^{(\fin,p)})$. Note that a nilpotent group $G$ is abelian if and only if $\Nil_2(G)$ is abelian. Therefore it suffices to show that $\overline{N}_2^{(\fin,p)}$ is abelian.

Since $p$ does not divide $|\overline{T}_2|$, we see that $\overline{N}_2^{(\fin,p)}$ intersects $\overline{T}_2$ trivially (inside $\overline{N}_2^{(\fin)}$) and hence projects isomorphically onto its image in $\overline{N}_2^{(\fin)}\slash \overline{T}_2$. But the latter group is abelian because
	$$\overline{N}_2^{(\fin)}\slash \overline{T}_2=\overline{N}_2\slash\langle \overline{N}_2^{(\text{tf})},\overline{T}_2\rangle$$
is a quotient of $\overline{N}_2\slash \overline{T}_2$, which is abelian because $\overline{N}_2$ is abelian modulo torsion (see Lemma \ref{lem:ab_mod_tor}). Therefore $\overline{N}_2^{(\fin,p)}$ is abelian, as desired.\end{proof}
	
We can now complete the proof of Theorem \ref{thm:almost_global}.
\begin{proof}[Proof of Theorem \ref{thm:almost_global}] We need to show that there is a finite set of primes $S$ and $c\geq 1$ such that for any $p\notin S$ and $k\geq 1$, the group $(\Gamma\slash \Gamma_k)^{(p)}$ is nilpotent of class at most $c$. Let $S$ be the set of divisors of $|\overline{T}_2|$. We will show that for $p\notin S$ and $k\geq 1$ arbitrary, $(\Gamma\slash\Gamma_k)^{(p)}$ is nilpotent of class at most 2, which will prove Theorem \ref{thm:almost_global}.

Let $c$ be the nilpotency class of $\Gamma\slash\Gamma_k$. Since $\Gamma\slash\Gamma_k$ is a quotient of $\overline{\Gamma}=\Gamma\slash\Gamma_\infty$, we we can realize $\Gamma\slash\Gamma_k$ as a quotient of $\Nil_c(\overline{\Gamma})$. By Lemma \ref{lem:nil2}.(i), we have $\Nil_c(\overline{\Gamma})=\overline{N}_c$. Then we must have
	$$(\Gamma\slash\Gamma_k)^{(p)} \cong (\overline{N}_c\slash\overline{\varphi}_k(\overline{N}_c))^{(p)}.$$
Since $\overline{N}_c$ is abelian modulo torsion, $\overline{N}_c\slash \langle \overline{\varphi}_k(\overline{N}_c), \overline{T}_c\rangle$ is abelian. Therefore the commutator subgroup $\gamma_1((\Gamma\slash \Gamma_k)^{(p)})$ is contained in the image of $\overline{T}_c^{(p)}$ in $\overline{N}_c\slash\varphi_k(\overline{N}_c)$. Hence to prove the claim it suffices to show that $\overline{T}_c^{(p)}$ is central in $\overline{N}_c$. 


Since $\overline{T}_c$ intersects $\overline{N}_c^{(\text{tf})}$ trivially, it suffices to show that $\overline{T}_c^{(p)}$ is central in $\overline{N}_c\slash \overline{N}_c^{(\text{tf})}$. To see this, write
	$$\overline{N}_2\slash \overline{N}_2^{(\text{tf})}\cong \prod_\ell (\overline{N}_c\slash \overline{N}_c^{(\text{tf})})^{(\ell)}$$
where $\ell$ runs over primes. Since $p$ does not divide $|\overline{T}_2|$, we know by Lemma \ref{lem:ab_or_div} that $(\overline{N}_c\slash \overline{N}_c^{(\text{tf})})^{(p)}$ is abelian, so $\overline{T}_c^{(p)}$ is indeed central.\end{proof}

\section{Local bound on nilpotency class}
\label{sec:local}

We will start by obtaining a local bound on nilpotency class of $\Gamma\slash\Gamma_k$ at a fixed prime $p$. Afterwards, we will finish the proof of Theorem \ref{thm:main}.
\begin{thmnr} Let $p$ be a prime. Then the nilpotency class of $(\Gamma\slash\Gamma_k)^{(p)}$ is bounded independently of $k$.
\label{thm:local}
\end{thmnr}
Before starting the proof, let us first prepare the stage. Recall that the embeddings $\varphi_k:\Gamma\hookrightarrow \Gamma$ induce open embeddings $\hphi_k:\hGamma\hookrightarrow \hGamma$, which by Theorem \ref{thm:reid_emb}, give rise to a $\hphi_k$-invariant decomposition
	$$\hGamma = C_k \rtimes Q_k.$$
Further we can write
	$$C_k=U_k\times S_k$$
where $U_k$ is a compact open subgroup of a product of unipotent $p$-adic Lie groups, and $S_k$ is a solvable group of finite exponent. Of course $U_k$ is nilpotent. By Proposition \ref{prop:con_nilp}, we also know that $S_k$ is pronilpotent. Pronilpotent groups split as a product over primes (see Proposition \ref{prop:pronilp_prod}), so we can write
	$$C_k = \prod_p C_k^{(p)}$$
where $C_k^{(p)}$ is pro-p.

Let us sketch the proof of the bound of Theorem \ref{thm:local}. Fix a prime $p$. We obtain some initial control on $(\Gamma\slash\Gamma_k)^{(p)}$  by showing that for sufficiently large $K$, the obvious quotient map $\hGamma\to\Gamma\slash\Gamma_k$ restricts to a surjection of $C_K^{(p)}$ onto $(\Gamma\slash\Gamma_k)^{(p)}$ (Lemma \ref{lem:local_surj} below). 

In particular, $(\Gamma\slash\Gamma_k)^{(p)}$ is a quotient of $C_K^{(p)}$. If $C_K^{(p)}$ is nilpotent rather than pronilpotent (e.g. if $S_K=1$), then this immediately yields a uniform bound on the nilpotency class of $(\Gamma\slash\Gamma_k)^{(p)}$, which proves Theorem \ref{thm:local}. To obtain such a bound in the general case where $C_K^{(p)}$ is merely pronilpotent, a more careful analysis of the image of $S_K^{(p)}$ in $\Gamma\slash\Gamma_k$ is needed. More precisely, we will see that modulo the image of $U_K^{(p)}$, there are only finitely many (necessarily nilpotent) possibilities for the image of $S_K^{(p)}$, independent of $k$. This will enable us to complete the proof of Theorem \ref{thm:local}. Let us now carry out this strategy.

\begin{lemnr} Fix a prime $p$. Then for sufficiently large $K$, the quotient map $\hGamma\to \Gamma\slash\Gamma_k$ restricts to a surjection $C_K^{(p)}\twoheadrightarrow (\Gamma\slash\Gamma_k)^{(p)}.$ 
\label{lem:local_surj}
\end{lemnr}
\begin{proof} Let us study the abelianization $(\Gamma\slash\Gamma_k)^{(p),\ab}$ of $(\Gamma\slash\Gamma_k)^{(p)}$. Clearly $(\Gamma\slash\Gamma_k)^{(p),\ab}$ surjects onto $(\Gamma\slash\Gamma_l)^{(p),\ab}$ whenever $k\geq l$. Hence rank$(\Gamma\slash\Gamma_k)^{(p),\ab}$ is nondecreasing in $k$.

On the other hand, since $\Gamma$ is finitely generated, we have
	$$\text{rank}(\Gamma\slash\Gamma_k)^{(p),\ab}\leq \text{rank}\,\Gamma^{\ab}<\infty$$
for any $k$. Therefore rank$(\Gamma\slash\Gamma_k)^{(p),\ab}$ is eventually constant. Let $K\geq 1$ such that $\rk(\Gamma\slash\Gamma_K)^{(p),\ab}$ is maximal. We claim that for any $k\geq 1$, the composition
	$$C^{(p)}_K\hookrightarrow \hGamma \twoheadrightarrow (\Gamma\slash\Gamma_k)^{(p)}$$
is surjective.

This is clear for $k\leq K$ because $\Gamma\slash\Gamma_k$ is a quotient of 		
	$$\Gamma\slash\Gamma_K\cong C_K\slash \hphi_K(C_K),$$
and hence $(\Gamma\slash\Gamma_k)^{(p)}$ is a quotient of $C_K^{(p)}\slash \hphi_K(C_K^{(p)})$.

Let $k>K$. Consider the image $H$ of the composition
	$$C^{(p)}_K\hookrightarrow \hGamma \twoheadrightarrow (\Gamma\slash\Gamma_k)^{(p)}.$$
We argue by contradiction, so assume that $H\subsetneq (\Gamma\slash\Gamma_k)^{(p)}$. Note that $H$ is normal in $(\Gamma\slash\Gamma_k)^{(p)}$ because $C_K^{(p)}$ is normal in $\hGamma$. The quotient
	$$F:=(\Gamma\slash\Gamma_k)^{(p)}\slash H$$
is a nontrivial $p$-group and hence there is a nontrivial map $F\to\bbF_p$. Consider the composition
	$$f: (\Gamma\slash\Gamma_k)^{(p)}\to F \to \bbF_p.$$
Since $(\Gamma\slash\Gamma_k)^{(p),\ab}$ and $(\Gamma\slash\Gamma_K)^{(p),\ab}$ have the same rank, any map $(\Gamma\slash\Gamma_k)^{(p),\ab}\to \bbF_p$ factors through 
	$$(\Gamma\slash\Gamma_k)^{(p),\ab} \to (\Gamma\slash\Gamma_K)^{(p),\ab}.$$
On the one hand, $f$ vanishes on the image $H$ of $C_K^{(p)}$ in $(\Gamma\slash\Gamma_k)^{(p)}$. On the other hand, $C_K^{(p)}$ surjects onto $(\Gamma\slash\Gamma_K)^{(p)}$. It follows that $f$ is trivial. This is a contradiction.\end{proof}

To complete the proof of Theorem \ref{thm:local} in the case that $S_K^{(p)}$ is not nilpotent but merely solvable, we need the following finiteness result for finite solvable groups. It is easily proved by induction on the length of the derived series.

\begin{lemnr} Let $r,d,N\geq 1$. There are only finitely many solvable groups that are generated by at most $r$ elements, have exponent at most $N$, and whose derived series has length at most $d$.
\label{lem:fin_solv}\end{lemnr}
%
%

We can now finish the proof of Theorem \ref{thm:local}.

\begin{proof}[Proof of Theorem \ref{thm:local}] Fix a prime $p$. We need to bound the nilpotency class of $(\Gamma\slash\Gamma_k)^{(p)}$ independently of $k$. Choose $K\geq 1$ as in Lemma \ref{lem:local_surj}, i.e. such that $C_K^{(p)}$ surjects onto $(\Gamma\slash\Gamma_k)^{(p)}$. Recall that $C_K=U_K\times S_K$ where $U_K$ is nilpotent and $S_K$ is a pronilpotent solvable group with bounded exponent.

Let $k\geq 1$ and let $H_U$ be the image of $U_K^{(p)}$ in $(\Gamma\slash\Gamma_k)^{(p)}$. Likewise let $H_S$ be the image of $S_K^{(p)}$. Note that $H_U$ is a normal subgroup, and Lemma \ref{lem:fin_solv} applies to the quotient 
	$$(\Gamma\slash\Gamma_k^{(p)})\slash H_U\cong H_S\slash (H_U\cap H_S).$$
Namely its number of generators is at most the number of generators of $\Gamma$, and its exponent and length of derived series are bounded by those of $S_K^{(p)}$. Hence there are only finitely many possibilities (independently of $k$) for the isomorphism type of
	$$H_S\slash (H_U\cap H_S).$$
Further since $S_K^{(p)}$ is pronilpotent, each of the finitely many options for $H_S\slash (H_U\cap H_S)$ is nilpotent, say of class at most $c_2$. 

Finally, let $c_1$ be the nilpotency class of $U_K^{(p)}$, so that $H_U$ is nilpotent of class at most $c_1$. Since $H_U$ and $H_S$ are nilpotent subgroups of $(\Gamma\slash\Gamma_k)^{(p)}$ that are mutually centralizing and that together generate $(\Gamma\slash\Gamma_k)^{(p)}$, we easily see that $(\Gamma\slash\Gamma_k)^{(p)}=\langle H_U, H_S\rangle$ can be written as a central extension
	$$1\to H_U\cap H_S \to \langle H_U, H_S\rangle \to \left(H_U\slash(H_U\cap H_S)\right) \times \left(H_S\slash(H_U\cap H_S)\right)\to 1.$$
In particular $(\Gamma\slash\Gamma_k)^{(p)}$ is nilpotent of class at most $1+\max\{c_1,c_2\}$.\end{proof}

\section{Proofs of main results}
\label{sec:proofs}

Having obtained a uniform bound on the nilpotency class of $\Gamma\slash\Gamma_k$ by Theorems \ref{thm:almost_global} and \ref{thm:local}, we complete the proof of Theorem \ref{thm:main}.

\begin{proof}[Proof of Theorem \ref{thm:main}] By Theorems \ref{thm:almost_global} and \ref{thm:local}, there exists $c\geq 1$ such that $\Gamma\slash\Gamma_k$ is nilpotent of class at most $c$, and hence so is $\overline{\Gamma}$. Again set $N_c:=\Gamma\slash\gamma_c(\Gamma)$, so $\overline{\Gamma}$ is a quotient of $N_c$. More precisely, $\varphi_k$ descend to maps 
	$$\overline{\varphi}_k:N_c\to N_c$$
and setting 
	$$\overline{N}_c:=N_c\slash \left(\cap_k \overline{\varphi}_k(N_c)\right),$$
we have $\overline{\Gamma}\cong \overline{N}_c$. But using Theorem \ref{thm:nilp_case}, the nilpotent case of Theorem \ref{thm:main}, we have already shown that $\overline{N}_c$ is abelian modulo torsion (see Lemma \ref{lem:ab_mod_tor}). Hence $\overline{\Gamma}$ is abelian modulo torsion, as desired.\end{proof}

We prove Corollary \ref{cor:char} showing that finite-index characteristic subgroups of $\Gamma$ that are isomorphic to $\Gamma$, come from free abelian quotients.
\begin{proof}[Proof of Corollary \ref{cor:char}]  Choose an embedding $\varphi:\Gamma\hookrightarrow \Gamma$ with image $\Gamma'$. An easy argument by induction shows that $\Gamma_n:=\varphi^n(\Gamma)$ is normal for all $n\geq 0$: Indeed, if $\varphi^n(\Gamma)$ is normal in $\Gamma$, then $\Gamma$ acts on $\varphi^n(\Gamma)$ by automorphisms. But $\varphi^{n+1}(\Gamma)$ is characteristic in $\varphi^n(\Gamma)$, and hence is $\Gamma$-invariant. The result now follows immediately from Theorem \ref{thm:main}.\end{proof} 

Finally we establish Corollary \ref{cor:scaleinv} showing that if $\Gamma$ is scale-invariant and the chain of subgroups $(\Gamma_n)_n$ consists of normal subgroups, then $\Gamma$ is virtually abelian.

\begin{proof}[Proof of Corollary \ref{cor:scaleinv}] Set $\Gamma_\infty:=\cap_n \Gamma_n$ as before. By Theorem \ref{thm:main}, we have that $\Gamma\slash\Gamma_\infty$ is abelian modulo torsion. Therefore there is a finite-index subgroup $\Lambda\subseteq \Gamma$ and some abelian group $A$ such that $\Lambda$ is an extension
	$$1\to\Gamma_\infty \to \Lambda\to A\to 1.$$
Since $\Lambda$ is residually finite (because it is abelian-by-finite) and $\Gamma_\infty$ is finite (by scale-invariance of $\Gamma$), there exists a finite-index subgroup $\Delta\subseteq\Lambda$ such that $\Delta\cap\Gamma_\infty=1$ and hence $\Delta$ is abelian.\end{proof}

\bibliographystyle{alpha}
\bibliography{noncohopf}

\end{document}